\documentclass[12pt]{article}
\usepackage{cite}
\usepackage{amssymb}
\usepackage{amsmath}
\usepackage{graphicx}
\newtheorem{thm}{Theorem}[section]

\newtheorem{lem}[thm]{Lemma}

\newtheorem{defn}[thm]{Definition}

\usepackage{graphics}
\numberwithin{equation}{section}

\newcommand{\eh}{\hfill}\newlength{\sperr}

\newenvironment{proof}{{\settowidth{\sperr}{\bf\rm
Proof}%
\par\addvspace{0.3cm}\noindent\parbox[t]{1.3\sperr}
{\textit{ P\eh r\eh o\eh o\eh f\eh .}}%
}}{\nopagebreak\mbox{}
$\Box$\par\addvspace{0.3cm}}

\def\ve{\varepsilon}
\def\nn{\nonumber}

\newtheorem{Pa}{Paper}[section]

\newtheorem{Rk}[Pa]{{\bf Remark}}

\title{Convolution type form of the Ito representation
of the infinitesimal generator  for Levy processes}
\author{Lev Sakhnovich}
\date{}
\begin{document}
\maketitle

\thanks{99 Cove ave., Milford, CT, 06461, USA \\
 E-mail: lsakhnovich@gmail.com}\\

 
\begin{abstract} In the present paper we show that the  Ito representation
of the infinitesimal generator $L$ for Levy processes can be written in a
convolution type form. Using the obtained convolution form and the theory
of integral equations with difference kernels we study the properties
of Levy processes.
  \end{abstract}

\section{Main notions}\label{sec1}

 Let us introduce  the notion of the Levy processes.

\begin{defn}\label{Definition 1.1.}
A stochastic process $\{X_{t}:t{\geq}0\}$ is called Levy
process ,if the following conditions are fulfilled:\\
1. Almost surely $X_{0}=0$, i.e. $P(X_{0}=0)=1$.\\
(One says that an event happens almost surely (a.s.) if it happens with probability one.)\\
2. For any $0{\leq}t_{1}<t_{2}...<t_{n}<\infty$ the random variables \\
 $X_{t_{2}}-X_{t_{1}}, X_{t_{3}}-X_{t_{4}},..., X_{t_{n}}-X_{t_{n-1}}$\\ are independent (independent increments).\\
( To call the increments of the process $X_{t}$ \emph{independent} means that
increments $X_{t_{2}}-X_{t_{1}}, X_{t_{3}}-X_{t_{4}},..., X_{t_{n}}-X_{t_{n-1}}$ are mutually  (not just
pairwise) independent.)\\
3. For any $s<t$ the distributions of   $X_{t}-X_{s}$ and $X_{t-s}$ are equal
(stationary increments).\\
4. Process $X_{t}$ is almost surely right continuous with left limits.\\
 Then Levy-Khinchine formula gives (see \cite{Bert}, \cite{Sato})
 \end{defn}
\begin{equation}
\mu(z,t)=E\{\mathrm{exp}[izX_{t}]\}=
\mathrm{exp}[-t\lambda(z)],\quad t{\geq}0,  \label{1.1} \end{equation}
where
\begin{equation}
\lambda(z)=\frac{1}{2}Az^{2}-i{\gamma}z-\int_{-\infty}^{\infty}(e^{ixz}-1-ixz1_{|x|<1})\mu(dx).
 \label{1.2} \end{equation} Here $A{\geq}0,\quad \gamma=\overline{\gamma},\quad z=\overline{z}$ and
 $\mu(dx)$ is a measure on the axis $(-\infty,\infty)$
  satisfying the conditions
\begin{equation}
\int_{-\infty}^{\infty}\frac{x^{2}}{1+x^{2}}\mu(dx)<\infty.
 \label{1.3} \end{equation}
The  Levy-Khinchine formula is determined by the Levy-Khinchine triplet
(A,$\gamma, \mu(dx)$).\\
By $P_{t}(x_{0},\Delta)$ we denote the probability
$P(X_{t}{\in}\Delta)$ when $P(X_{0}=x_{0})=1$ and $\Delta{\in}R$.
The transition operator $P_{t}$ is defined by the formula
\begin{equation}
P_{t}f(x)=\int_{-\infty}^{\infty}P_{t}(x,dy)f(y). \label{1.4} \end{equation} Let
$C_{0}$ be the Banach space of continuous functions $f(x)$ ,
satisfying the condition $\mathrm{lim}f(x)=0,\quad |x|{\to}\infty$
with the norm $||f||=\mathrm{sup}_{x}|f(x)|$. We denote by
$C_{0}^{n}$ the set of $f(x){\in}C_{0}$ such that
$f^{(k)}(x){\in}C_{0},\quad (1{\leq}k{\leq}n).$ It is known that
\cite{Sato}
\begin{equation} P_{t}f{\in}C_{0}, \label{1.5} \end{equation}
if $f(x){\in}C_{0}^{2}.$\\
Now we formulate the following important result (see \cite{Sato}) .

\begin{thm}\label{Theorem 1.1.} \textbf{(Levy-Ito decomposition.)}
{The family of the operators $P_{t}\quad
(t{\geq}0)$ defined by the Levy process $X_{t}$ is a strongly
continuous  semigroup on $C_{0}$ with the norm $||P_{t}||=1$. Let
$L$ be its infinitesimal generator. Then}
\begin{equation} Lf=\frac{1}{2}A\frac {d^{2}f}{dx^{2}}+\gamma
\frac{df}{dx}+\int_{-\infty}^{\infty}(f(x+y)-f(x)-y\frac{df}{dx}1_{|y|<1})\mu(dy), \label{1.6} \end{equation}
{where} $f{\in}C_{0}^{2}$.
\end{thm}

\section{Convolution type form of infinitesimal generator}
\label{sec2}\paragraph{1.}
In this section we prove that  the
infinitesimal generator $L$ can be represented in the special
convolution type form
\begin{equation}
Lf=\frac{d}{dx}S\frac{d}{dx}f, \label{2.1}
\end{equation}
where the operator $S$ is defined by the relation
\begin{equation}
Sf=\frac{1}{2}Af+\int_{-\infty}^{\infty}k(y-x)f(y)dy. \label{2.2} \end{equation}
We note that  for arbitrary $a\, (0<a<\infty)$ the inequality
\begin{equation} \int_{-a}^{a}|k(t)|dt<\infty \label{2.3} \end{equation}
is true.\\ Formula \eqref{2.2} was proved  before in our works \cite{Sakh6} under some additional
conditions. In the present paper we omit these additional conditions and prove the formula \eqref{2.2} for the general case.The representation of $L$ in form \eqref{2.1} is convenient as
the operator $L$ is expressed with the help of the classic
differential and convolution operators. Using the obtained
convolution form of the generator $L$ and the theory of integral equations
with difference kernels \cite{Sakh5} we investigate the properties of a wide class
of Levy processes.\\By $C(a)$ we denote the
set of functions $f(x){\in}C_{0}$ which have the following property:
 \begin{equation} f(x)=0 ,\quad x{\notin}[-a,a]\label{2.4} \end{equation}
i.e. the function $f(x)$ is equal to zero in the neighborhood of $x=\infty$.
 We note, that parameter $a$
can be different for different $f$.\\
We introduce the functions
\begin{equation}\mu_{-}(x)=\int_{-\infty}^{x}\mu(dx),\,x<0,\label{2.6} \end{equation}
\begin{equation}\mu_{+}(x)=-\int_{x}^{\infty}\mu(dx),\,x>0,\label{2.5} \end{equation}
where the functions $\mu_{-}(x)$ and $\mu_{+}(x)$ are  monotonically increasing on the  half-axis   $(-\infty,0]$ and $[0,\infty)$ respectively and
\begin{equation}\mu_{+}(x){\to}0,\,x{\to}+\infty;\,
\mu_{-}(x){\to}0,\,x{\to}-\infty.\label{2.6a}\end{equation}
We note that
\begin{equation}\mu_{-}(x){\geq}0,\,x<0;\,\mu_{+}(x){\leq}0,\,x>0.\label{2.6b}\end{equation} In view of \eqref{1.3}
the integrals in the right sides of \eqref{2.5} and \eqref{2.6} are convergent.
Hence we have\\
\begin{equation}\int_{-\infty}^{\infty}f(x)\mu(dx)=\int_{-\infty}^{0}f(x)d\mu_{-}(x)+
\int_{0}^{\infty}f(x)d\mu_{+}(x).\label{2.6c}\end{equation}
\begin{thm}\label{Theorem 2.1} The following relations
\begin{equation}{\ve}^{2}\mu_{\pm}(\pm{\ve}){\to}0,\,\ve{\to}+0,\label{2.7} \end{equation}
\begin{equation}-\int_{-a}^{0}x\mu_{-}(x)dx<\infty,\,
-\int_{0}^{a}x\mu_{+}(x)dx<\infty,\,0<a<\infty\label{2.8} \end{equation}
are true\end{thm}
\begin{proof} According to \eqref{1.3} we have
\begin{equation}0{\leq}\int_{-a}^{-\ve}x^{2}d\mu_{-}(x){\leq}M,
\label{2.9}\end{equation}where $M$ does not depend from $\ve$.
Integrating by parts the integral of \eqref{2.9} we  obtain:
\begin{equation}\int_{-a}^{-\ve}x^{2}d\mu_{-}(x)=\ve^{2}\mu_{-}(-\ve)-a^{2}\mu_{-}(-a)
-2\int_{-a}^{-\ve}x\mu_{-}(x)dx{\leq}M,\label{1}\end{equation} The function $-\int_{-a}^{-\ve}x\mu_{-}(x)dx$ of $\ve$
is monotonic increasing. According
to \eqref{1} this function is bounded. Hence we have
\begin{equation}\lim_{\ve{\to}+0}\int_{-a}^{-\ve}x\mu_{-}(x)dx=\int_{-a}^{0}x\mu_{-}(x)dx
\label{2}\end{equation} It follows from \eqref{1} and \eqref{2} that
\begin{equation}\lim_{\ve{\to}+0}\ve^{2}\mu_{-}(-\ve)=0.\label{3}\end{equation}
Thus, relations \eqref{2.7} and \eqref{2.8} are proved for $\mu_{-}(x)$. In the same way
relations \eqref{2.7} and \eqref{2.8} can be proved  for $\mu_{+}(x)$.
\end{proof}
\paragraph{2.}Let us introduce the functions
\begin{equation}k_{-}(x)=\int_{-1}^{x}\mu_{-}(t)dt,\,-\infty{\leq}x<0,\label{2.11}
\end{equation}
\begin{equation}k_{+}(x)=-\int_{x}^{1}\mu_{+}(t)dt,\,0<x{\leq}+\infty.\label{2.12}
\end{equation}
 In view of \eqref{2.8}
the integrals in the right sides of \eqref{2.11} and \eqref{2.12} are absolutely convergent.
From \eqref{2.11} and \eqref{2.12} we obtain the assertions:
\begin{thm}\label{Theorem 2.2}
1.The function $k_{-}(x)$  is monotonically increasing on the   $(-\infty,0)$ and
\begin{equation}k_{-}(x){\geq}0,\,-1{\leq}x<0.\label{2.13}\end{equation}
2.The function $k_{+}(x)$  is monotonically decreasing on the   $(0,+\infty)$ and
\begin{equation}k_{+}(x){\geq}0,\,0<x{\leq}1.\label{2.14}\end{equation}
\end{thm}
Further we need the following result:
\begin{thm}\label{Theorem 2.3} The following relations
\begin{equation}{\ve}k_{-}(-\ve){\to}0,\,\ve{\to}+0;\,{\ve}k_{+}(\ve){\to}0,\,\ve{\to}+0;\label{2.15}\end{equation}
\begin{equation}\int_{-1}^{0}k_{-}(x)dx<\infty;\,\int_{0}^{1}k_{+}(x)dx<\infty
\label{2.16}\end{equation}are valid.
\end{thm}
\begin{proof} According to \eqref{2.8} we have
\begin{equation}0{\leq}-\int_{-1}^{-\ve}x\mu_{-}(x)dx{\leq}M,,
\label{5}\end{equation}where $M$ does not depend from $\ve$.
Integrating by parts the integral  of \eqref{5} we  obtain:
\begin{equation}-\int_{-1}^{-\ve}x\mu_{-}(x)dx={\ve}k_{-}(-\ve)-k_{-}(-1)
+\int_{-1}^{-\ve}k_{-}(x)dx{\leq}M,\label{6}\end{equation}
The function $\int_{-1}^{-\ve}k_{-}(x)dx$ of $\ve$
is monotonic increasing.This function is bounded (see \eqref{6}) . Hence we have
\begin{equation}\lim_{\ve{\to}+0}\int_{-1}^{-\ve}k_{-}(x)dx=\int_{-1}^{0}k_{-}(x)dx
\label{7}\end{equation} It follows from \eqref{6} and \eqref{7} that
\begin{equation}\lim_{\ve{\to}+0}{\ve}k_{-}(-\ve)=0.\label{8}\end{equation}
Thus, relations \eqref{2.15} and \eqref{2.16} are proved for $k_{-}(x)$. In the same way
relations \eqref{2.15} and \eqref{2.16} can be proved  for $k_{+}(x)$.
\end{proof}

\paragraph{3.}
We use  the following notation
\begin{equation}J(f)=J_{1}(f)+J_{2}(f),\label{2.17}\end{equation}where
\begin{equation}
J_{1}(f)=\frac{d}{dx}\int_{-\infty}^{x}f^{\prime}(y)k_{-}(y-x)dy,\quad  f(x){\in}C(a), \label{2.18} \end{equation}
\begin{equation}
J_{2}(f)=\frac{d}{dx}\int_{x}^{\infty}f^{\prime}(y)k_{+}(y-x)dy, \quad  f(x){\in}C(a). \label{2.19!} \end{equation}
\begin{lem}\label{Lemma 1.}
The operator $J(f)$ defined by \eqref{2.17} can be represented in the form
\begin{equation}
J(f)=\int_{-\infty}^{\infty}[f(y+x)-f(x)-y\frac{df(x)}{dx}1_{|y|{\leq}1}]{\mu}(dy)
+{\Gamma}f^{\prime}(x),
 \label{1.2.17} \end{equation}
 where $\Gamma=\overline{\Gamma}$ and
$f(x){\in}C(a)$.
\end{lem}

\begin{proof}
From \eqref{2.18}  we obtain  the relation
\begin{equation}
J_{1}(f)=-\int_{x-1}^{x}[f^{\prime}(y)-f^{\prime}(x)]
k_{-}^{\prime}(y-x)dy-
\int_{-a}^{x-1}f^{\prime}(y)k_{-}^{\prime}(y-x)dy. \label{2.19} \end{equation}
 By proving \eqref{2.19} we used relations
 \eqref{2.11},\eqref{2.15} and equality
 \begin{equation}\int_{x-1}^{x}k_{-}(y-x)dy=\int_{-1}^{0}k_{-}(v)dv.
 \label{2.20}\end{equation}We introduce the notations \begin{equation}
P_{1}(x,y)=f(y)-f(x)-(y-x)f^{\prime}(x),\quad
P_{2}(x,y)=f(y)-f(x). \label{2.21} \end{equation} Using notations \eqref{2.21}
we represent \eqref{2.19} in the form
\begin{equation}J_{1}(f)=-\int_{-1}^{0}\frac{\partial}{\partial{y}}P_{1}(x,y+x)\mu_{-}(y)dy
-\int_{-a-x}^{-1}\frac{\partial}{\partial{y}}P_{2}(x,y+x)\mu_{-}(y)dy, \label{2.22} \end{equation}Integrating by parts the integrals of \eqref{2.22}
 we deduce that
\begin{align}\nn J_{1}(f)=&f^{\prime}(x)\gamma_{1}+\int_{-1}^{0}P_{1}(x,y+x)d\mu_{-}(y)
+\int_{-a-x}^{-1}P_{2}(x,y+x)d\mu_{-}(y)
\\ &
+P_{2}(x,-a)\mu_{-}(-a), \label{2.23} \end{align}
where $\gamma_{1}=k_{-}^{\prime}(-1).$ It follows from \eqref{1.3} that the integrals in \eqref{2.23} are absolutely convergent. Passing to the limit in \eqref{2.23},when $a{\to}+\infty$, and taking into account  \eqref{2.19} , \eqref{2.20} we have
 \begin{equation}
J_{1}(f)=\int_{-\infty}^{x}[f(y+x)-f(x)-y\frac{df(x)}{dx}1_{|y|{\leq}1}]d{\mu_{-}}(y)
+{\gamma}_{1}f^{\prime}(x). \label{1.2.21} \end{equation}
In the same way it can be
proved that \begin{equation}
J_{2}(f)=\int_{x}^{\infty}[f(y+x)-f(x)-y\frac{df(x)}{dx}1_{|y|{\leq}1}]d{\mu_{+}}(y)
+{\gamma}_{2}f^{\prime}(x), \label{1.2.22} \end{equation}where
$\gamma_{2}=k_{+}^{\prime}(1).$ The relation \eqref{1.2.17} follows directly
from  \eqref{1.2.21} and \eqref{1.2.22}. Here $\Gamma=\gamma_{1}+\gamma_{2}.$ The lemma is
proved.
\end{proof}
\begin{Rk}\label{Remark 2.1.}
The operator $L_{0}f=\frac{d}{dx}f$ can be
represented in  form \eqref{2.1}, \eqref{2.2}, where \begin{equation}
S_{0}f=\int_{-\infty}^{\infty}p_{0}(x-y)f(y)dy, \label{2.23!} \end{equation}
\begin{equation}p_{0}(x)=\frac{1}{2}\,\mathrm{sign}(x). \label{2.24} \end{equation}
\end{Rk}

From Lemmas \ref{Lemma 1.},  and Remark \ref{Remark 2.1.} we deduce the following
assertion.

\begin{thm}\label{Theorem 2.4.}

The infinitesimal generator  $L$
has a
convolution type form \eqref{2.1},  \eqref{2.2}.
\end{thm}

\end{document}